\documentclass[12pt]{amsart}
\usepackage{amssymb,latexsym}
\usepackage{enumerate}
\usepackage{hyperref}
\usepackage{fullpage}
\usepackage{tikz}

\usepackage{fixme}
\fxsetup{
    status=draft,
    author=,
    layout=margin,
    theme=color
}

\makeatletter \@namedef{subjclassname@2010}{%
  \textup{2010} Mathematics Subject Classification}
\makeatother

\newcounter{thm} \numberwithin{thm}{section}
\newtheorem{Theorem}[thm]{Theorem}

\newtheorem{Lemma}[thm]{Lemma}

\newtheorem*{Definition}{Definition}

\newtheorem*{Claim}{Claim}

\newcommand{\R}{\mathbb{R}}

\author[O. Roche-Newton]{Oliver Roche-Newton} \address{Johann Radon Institute for Computational and Applied Mathematics\\
Linz, Austria}
\email{o.rochenewton@gmail.com}

\date{}

\begin{document}

\baselineskip=17pt

\title{Sums, products and dilates on sparse graphs}

\date{}
\maketitle

\begin{abstract}
Let $A \subset \mathbb R$ and $G \subset A \times A$. We prove that, for any $\lambda \in \mathbb R \setminus \{-1,0,1\}$,
\[
\max \{|A+_G A|, |A+_G \lambda A|, |A\cdot_G A|\} \gg |G|^{6/11}.
\]
\end{abstract}

\section{Introduction}
Given a finite set $A \subset \mathbb R$ and a set $G \subset A \times A$, the sum set of $A$ restricted to $G$ is the set
\[
A+_G A:= \{ a+b : (a,b)\in G \}.
\]
The restricted product set $A\cdot_G A= \{ab: (a,b) \in G\}$ is defined similarly. One may consider the question of finding lower bounds for the quantity
\[
\max \{|A+_G A|, |A\cdot_G  A|\}.
\]
The case when $G=A \times A$ corresponds to the classical sum-product problem over $\mathbb R$, where the current state of the art, due to Rudnev and Stevens \cite{RS} is that the bound\footnote{Here and throughout this paper, the notation
 $X\gg Y$, $Y \ll X,$ $X=\Omega(Y)$, and $Y=O(X)$ are all equivalent and mean that $X\geq cY$ for some absolute constant $c>0$. $X \approx Y$ and $X=\Theta (Y)$ denote that both $X \gg Y$ and $X \ll Y$ hold. $X \gg_a Y$ means that the implied constant is no longer absolute, but depends on $a$. We also use the notation $X \gtrsim Y $ and $Y \lesssim X$ to denote that $X \gg Y/(\log_2 Y)^c$ for some absolute constant $c>0$.}
\[
\max \{|A+ A|, |A\cdot A|\} \gg |A|^{\frac{4}{3}+\frac{2}{1167}-o(1)},
\]
holds for all $A \subset \mathbb R$. 

For an arbitrary $G \subset A \times A$, the bound
\begin{equation} \label{trivial}
\max \{|A+_G A|, |A\cdot_G A|\} \geq \frac{1}{\sqrt 2}|G|^{1/2}.
\end{equation}
holds by a simple argument.\footnote{
If $|A+_G A| \leq \frac{1}{\sqrt 2}|G|^{1/2}$ then by the pigeonhole principle there is some $x$ such that $|\{(a,b) \in G : a+b=x\}|=t \geq \sqrt 2|G|^{1/2}$.
 We have
$
x=a_1+b_1=\dots=a_t+b_t$, $ (a_i,b_i) \in G$
But then in the multiset of products $\{a_ib_i : 1 \leq i \leq t\}$, each element occurs with multiplicity at most 2. This implies that $|A \cdot_G A|\geq \frac{t}{2} \geq \frac{1 }{\sqrt 2}|G|^{1/2}$.
}

This trivial bound cannot be improved in general, as there exists $A \subset \mathbb R$ and $G \subset A \times A$ such that
\[
\max \{|A+_G A|, |A\cdot_G A|\} \ll |G|^{1/2}.
\]
Consider for instance the following construction of Chang \cite{Chang}:
\begin{equation} \label{construction}
A= \{ \sqrt i \pm \sqrt j  : i,j \in [n]\},\,\,\,\, G=\{(\sqrt i + \sqrt j, \sqrt i - \sqrt j) : i,j \in [n] \}.
\end{equation}
The sets $A$ and $G$ have cardinality $\Theta (n^2)$. Meanwhile,
\[
A+_G A= \{2 \sqrt i : i \in [n]\}
\]
has cardinality $n$, and
\[
A\cdot_G A= \{i-j : i,j \in [n]\}
\]
has cardinality $O(n)$. In fact, a slightly different construction shows that the bound \eqref{trivial} is completely tight, including the multiplicative constant $\frac{1}{\sqrt 2}$. See the forthcoming section \ref{sec:picture}.

The common vernacular in these restricted sum-product problems is to view the set $G$ as the edge set of a bipartite graph on $A \times A$. The graph $G$ in example \eqref{construction} is fairly small, or \textit{sparse}. Non-trivial bounds for sufficiently dense graphs are known; see, for instance, the work of Alon, Ruzsa and Solymosi \cite[Theorem 10]{ARS}. 

\subsection{Introducing a third set}

Given the trivial lower bound \eqref{trivial} and the matching construction \eqref{construction}, it seems there is not much more to say about sum-product estimates over $\mathbb R$ along arbitrary (possibly sparse) graphs. One can modify the question by considering a third set, naturally occurring in sum-product type problems. Chang \cite{Chang} considered the question of finding a lower bound for
\[
\max \{|A+_G A|, |A-_G A|, |A\cdot_G A|\}
\]
where $A-_G A:=\{a-b: (a,b) \in G \}$. Chang gave a small improvement to the trivial bound for all $A \subset \mathbb Z$, proving that $\max \{|A+_G A|, |A-_G A|, |A\cdot_G A|\} \gg_{\epsilon} |G|^{1/2} (\log |G|)^{1/48 - \epsilon}$. However, she also showed that a non-trivial bound is impossible for $A \subset \mathbb R$. Indeed, for the same example \eqref{construction} above we have $A-_G A=\{2\sqrt j : j \in [n]\}$
and thus
\begin{equation} \label{3construction}
|A+_G A|,|A-_G A|, |A \cdot_G  A| \ll |G|^{1/2}.
\end{equation}

The aim of this paper is to show that a non-trivial bound can be obtained for a small perturbation of this problem. Consider the set
\[
A+_G \lambda A:=\{a+\lambda b : (a,b) \in G\},
\]
where $\lambda$ is some non-zero real number. The case $\lambda =-1$ corresponds to the restricted difference set $A-_G A$. The main result of this note shows that this value $-1$, giving rise to the situation in \eqref{3construction}, is special, and that a non-trivial bound holds for any other $\lambda$. 
\begin{Theorem} \label{thm:main}
Let $A,B \subset \mathbb R$, $G \subset A \times B$, and $\lambda \in \mathbb R \setminus \{-1,0,1\}$. Then
\[
\max \{|A+_G B|,|A+_G \lambda B|, |A \cdot_G B| \}\gg |G|^{6/11}.
\]
\end{Theorem}
Our primary interest is in the case when $A=B$, but note that the result also holds in the asymmetric case when sums and products of two different sets are considered. Furthermore, the result can be extended to the setting of $ \mathbb C$ by instead using a complex analogue of the forthcoming Theorem \ref{11/6ES}.

\subsection{Integers versus reals} \label{sec:picture}
The fact that a non-trivial bound holds for \linebreak $\max \{|A+_G A|, |A-_G A|, |A\cdot_G A|\}$ when $A \subset \mathbb Z$, but does not for $A \subset \mathbb R$, highlights an interesting distinction between the sum-product problems in the real and integer settings. A similar situation occurs for the problem of bounding $\max \{|A+_G A|, |A\cdot_G A|\}$. Conditional on the Uniformity Conjecture, the non-trivial bound
\begin{equation} \label{SS}
\max \{|A+_G B|, |A\cdot_G B|\} \gg |G|^{3/5},
\end{equation}
implicit in the work of Shkredov and Solymosi \cite{SS}, holds for any $A,B \subset \mathbb Z$ and $G \subset A \times B$. This builds on work of Alon et al. \cite{Alon}. See also \cite{MRNSW} for applications of the Uniformity Conjecture to sum-product problems on sparse graphs. 

The previous construction \eqref{construction} shows that the bound \eqref{SS} does not hold for all $A \subset \mathbb R$. However, one can build such sets by considering the following picture.

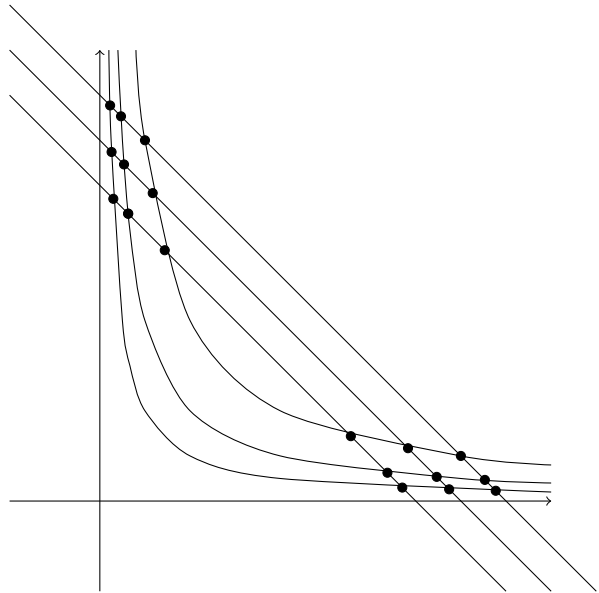
\begin{figure}[h!]
    \centering
       \begin{tikzpicture}[scale =1.2]
       \draw [->] (0,-1) -- (0,5);
\draw [->] (-1,0) -- (5,0);
\draw plot [smooth] coordinates { (0.1,5) (0.125,4) (1/4,2) (0.3333 ,1.5) (1/2,1) (1,1/2) (2,1/4) (5,0.1)  };
       \draw plot [smooth] coordinates { (1/5,5) (1/4,4) (0.3333 ,3) (1/2,2) (1,1) (2,1/2) (4,1/4) (5,1/5)  };
       \draw plot [smooth] coordinates { (0.4,5) (1/2,4) (1,2) (2,1) (4,1/2) (5,0.4)  };
       \draw (-1,5) -- (5,-1);
       \draw (-1,4.5) -- (4.5,-1);
       \draw (-1,5.5) -- (5.5,-1);
       
       \draw[fill] (3.870,0.130) circle [radius=0.05];
       \draw[fill] (0.130, 3.870) circle [radius=0.05];
       \draw[fill] (0.268, 3.732) circle [radius=0.05];
       \draw[fill] ( 3.732, 0.268) circle [radius=0.05];
       \draw[fill] ( 3.414, 0.58578) circle [radius=0.05];
       \draw[fill] ( 0.58578, 3.414) circle [radius=0.05];
       \draw[fill] ( 3.3507810593582, 0.14921894064179) circle [radius=0.05];
       \draw[fill] ( 0.14921894064179, 3.3507810593582) circle [radius=0.05];
        \draw[fill] ( 3.1861406616345, 0.31385933836549) circle [radius=0.05];
        \draw[fill] (  0.31385933836549, 3.1861406616345) circle [radius=0.05];
        \draw[fill] (  0.31385933836549, 3.1861406616345) circle [radius=0.05];
        \draw[fill] (  2.7807764064044, 0.71922359359558) circle [radius=0.05];
        \draw[fill] ( 0.71922359359558, 2.7807764064044) circle [radius=0.05];
        \draw[fill] ( 4.3860009363294, 0.11399906367062) circle [radius=0.05];
        \draw[fill] (  0.11399906367062, 4.3860009363294) circle [radius=0.05];
        \draw[fill] (  4.2655644370746, 0.23443556292536) circle [radius=0.05];
        \draw[fill] (  0.23443556292536, 4.265564437074) circle [radius=0.05];
        \draw[fill] (  4, 0.5) circle [radius=0.05];
        \draw[fill] ( 0.5, 4) circle [radius=0.05];

4.2655644370746
0.23443556292536

\end{tikzpicture}
    \caption{$n$ lines,  and $n$ hyperbolas of the form $xy=c$ (in this illustration we just show the positive quadrant). The intersection points make a set $G \subset A \times B$ with the property that $|A+_G B|, |A \cdot_G B| \ll |G|^{1/2}$.}
    \label{fig:my_label}
\end{figure}

This illustration shows $n$ parallel lines with slope $-1$, and $n$ hyperbolas with equation $xy=c$. Each line intersects each hyperbola in two points, and so these objects intersect pairwise to give a set $G$ of $2n^2$ points. We can write $G \subset A \times B$ for some $A,B \subset \mathbb R$. $A$ is simply the projection of the point set $G$ onto the $x$-axis, and likewise $B$ is the projection onto the $y$-axis. However, by construction,
\[
|A+_G B|=n. 
\]
This is because the elements of $A+_G B$ correspond precisely with the $n$ parallel lines in the picture. Similarly, the elements of $A \cdot_G B$ are in bijection with the $n$ hyperbolas. We conclude that
\[
|A+_G B|, |A\cdot_G B| = \frac{1}{\sqrt 2 }|G|^{1/2}.
\]

\subsection{Structure of the rest of this paper}
The next section describes the main tool to be used throughout this paper, which is a bound for the Elekes-Szab\'{o} problem from \cite{RSdZ}. In section \ref{sec:proof}, the proof of Theorem \ref{thm:main} is given. Section \ref{sec:construction} describes a non-trivial construction which shows that the exponent $6/11$ in Theorem \ref{thm:main} cannot be improved beyond $3/4$. Section \ref{sec:variants} considers some variants of the main problem using different combinations of three sets. We give a construction for one such problem, which shows that a non-trivial lower bound is impossible, and indicate a variant for which an analogue of Theorem \ref{thm:main} does hold. Finally, in section \ref{sec:SP}, we use a similar argument to prove a new bound for $\max \{ |A+_G A|, |A \cdot_G A|\}$. This result is non-trivial for $|G| \geq |A|^{3/2}$, a broader range of values than was previously known, and improves a result of Alon, Ruzsa and Solymosi \cite{ARS} for $|G| \leq |A|^{11/6}$.

\section{The Elekes-Szab\'{o} Problem} \label{sec:ES} 

The proof of Theorem \ref{thm:main} uses theory developed to tackle what we call Elekes-Szab\'{o} type problems, first considered in \cite{ER00} and \cite{ES}. Given a polynomial $F \in \mathbb R[x,y,z]$, let $Z(F)$ denote the zero set of $F$;
\[
Z(F):= \{(x,y,z) \in \mathbb R^3 : F(x,y,z)=0 \}.
\]
For arbitrary sets $A,B,C$ of the same cardinality $n$, Elekes and Szab\'{o} proved a non-trivial bound
\[
|Z(F) \cap A \times B \times C| \ll n^{2-c}
\]
for some positive constant $c$, provided that $F$ is \textit{non-degenerate}.

\begin{Definition} A polynomial $F \in \R[x,y,z]$ is \textit{degenerate} if there exists a one-dimensional sub-variety $Z_0 \subseteq Z(F)$ such that for all $v\in Z(F)\setminus Z_0$, there are open intervals $I_1,I_2,I_3\subseteq \R$ and injective real-analytic functions $\phi_i:I_i\rightarrow \R$ with real-analytic inverses ($i = 1,2,3$) such that $ v \in I_1\times I_2\times I_3$ and for any $(x,y,z)\in I_1\times I_2\times I_3$, we have
    \[ (x, y, z) \in Z(F) \text{ if and only if } \phi_1(x) + \phi_2(y) + \phi_3(z) = 0\,.\]
    Otherwise, we say that $F$ is \textit{non-degenerate}.

\end{Definition}

Some simple examples of degenerate polynomials are $F(x,y,z)= x +y +z$, $G(x,y,z)=xyz$ and $H(x,y,z)=x^2+y^3+z^{17}$. Meanwhile, for instance, $F(x,y,z)=(x-y)^2+x+z$ is non-degenerate (see \cite{MRNWdZ} for  a simple construction showing that this polynomial can have a relatively large intersection with a Cartesian product).

 Quantitative advances on the work of Elekes and Szab\'{o} have been attained in the intervening years, often motivated by applications in discrete geometry. We will make use of the following result of Raz, Sharir and de Zeeuw \cite{RSdZ}.

 \begin{Theorem} \label{11/6ES}
Let $F \in \mathbb R[x,y,z]$ be a non-degenerate irreducible polynomial of degree $d$, such that none of the partial derivatives $\frac{\partial F}{\partial x}, \frac{\partial F}{\partial y}, \frac{\partial F}{\partial z}$ vanish.
Then for all $A,B,C\subseteq \mathbb{R}$ we have
    \begin{equation}\label{eq:RSZ11/6}
    |Z(F) \cap (A \times B \times C)| = O_d( |A|^{1/2}|B|^{2/3}|C|^{2/3}+|A|^{1/2}(|A|^{1/2}+|B|+|C|)).
    \end{equation}

\end{Theorem}

In practice, it is not always easy to check that a given polynomial is non-degenerate. To help with this task, we will use an idea introduced by Elekes and R\'onyai \cite{ER00},
which is that non-degeneracy can be verified using the following derivative test; see \cite[Lemma 33]{ES} or \cite[Lemma 2.2]{Z18}.

\begin{Lemma}\label{lem:difftest}
Let $f:\R^2\to \R$ be a smooth function on some open set $U\subset \R^2$ with $f_x$ and $f_y$ not identically zero.
There exist smooth functions $\psi, \varphi_1,\varphi_2$ on $U$ such that 
\[ f(x,y) = \psi(\varphi_1(x) + \varphi_2(y)),\]
if and only if
\begin{equation}\label{eq:test} \frac{\partial^2\left(\log|f_x/f_y|\right)}{\partial x\partial y}
\end{equation}
is identically zero on $U$.
\end{Lemma}

In practice, Lemma \ref{lem:difftest} allows us to test the whether or not the polynomial $F(x,y,z)$ is degenerate by rearranging the expression $F(x,y,z)=0$ into the form $z=f(x,y)$ (if this is possible), computing the resulting expression \eqref{eq:test}, and checking whether it is identically zero on an open set.

For the case when the expression $F(x,y,z)=0$ can already be rearranged into the form $z=f(x,y)$ where $f$ is a polynomial, the definition of non-degeneracy becomes more straightforward. See \cite[Theorem 2]{RSS}. For some applications in this paper, \cite[Theorem 2]{RSS} would be sufficient, and non-degeneracy can instead be checked by some elementary combinatorial arguments which use less algebraic theory. However, there are some cases where the results of \cite{RSS} do not immediately give us what we need, and so we use Theorem \ref{11/6ES} throughout in order to present a unified approach.

\section{Proof of Theorem \ref{thm:main}} \label{sec:proof}

Write
\[
C:=A+_G B,\,\,\,\,D:=A+_G \lambda B, \,\,\,\, E:=A \cdot_G B.
\]
We will double count solutions to the system of equations
\begin{align}
c&=a+b \label{eq1}
\\ d&=a+ \lambda b \label{eq2}
\\ e&=ab, \label{eq3}
\end{align}
such that $(a,b) \in G$ and $(c,d,e) \in C \times D \times E$. Define $S$ to be the number of solutions to this system. A simple observation is that
\begin{equation} \label{easy}
S=|G|
\end{equation}
since each pair $(a,b) \in G$ gives rise to a unique solution.

We seek to find a complementary upper bound for $S$ via an application of Theorem \ref{11/6ES}. We can eliminate $a$ and $b$ from the system above. Together, \eqref{eq1} and \eqref{eq2} imply that
\[
b=\frac{1}{1-\lambda}(c-d), \,\,\,\,\,\,\,\,\, a=  \frac{1}{1-\lambda}(d-\lambda c).
\]
Note that the assumption that $\lambda \neq 1$ is used here to ensure that these expressions are well defined. Substituting this information into \eqref{eq3} gives
\begin{equation} \label{3var}
e=\left(\frac{1}{1-\lambda}\right)^2 (c-d)(d-\lambda c).
\end{equation}
Define 
\[F(X,Y,Z)=\left(\frac{1}{1-\lambda}\right)^2 (X-Y)(Y-\lambda X)-Z.
\]
We have thus deduced that every contribution $(a,b,c,d,e)$ to $S$ gives rise to a solution $(c,d,e)$ to the equation $F(c,d,e)=0$. Furthermore, no triple $(c,d,e)$ contributes more than once to $S$, since for fixed $c$ and $e$ there exists at most one pair $(a,b)$ which satisfies both \eqref{eq1} and \eqref{eq3}. Therefore, we have
\begin{equation} \label{SandF}
S \leq |Z(F) \cap C \times D \times E|.
\end{equation}

Note also that $F$ is irreducible since the $Z$ term appears in isolation.

\begin{Claim}
$F$ is non-degenerate.
\end{Claim}

Assuming that the claim is correct, we can apply Theorem \ref{11/6ES} to obtain the upper bound
\[
|Z(F) \cap C \times D \times E| \ll |C|^{1/2}|D|^{2/3}|E|^{2/3} + |C|^{1/2}(|C|^{1/2}+|D|+|E|).
\]
Combining this with \eqref{SandF} and \eqref{easy}, we have
\[
|G| \ll |C|^{1/2}|D|^{2/3}|E|^{2/3} + |C|^{1/2}(|C|^{1/2}+|D|+|E|)
\]
and it follows that
\[
\max \{|C|,|D|,|E|\} \gg |G|^{6/11},
\]
as required. It remains to prove the claim.

\begin{proof}[Proof of Claim]

Proof by contradiction. Suppose that $F(X,Y,Z) $ is degenerate. 
Then there is some open neighbourhood $I_1\times I_2\times I_3$ intersecting $Z(F)$, such that $F(X,Y,Z) = 0$ if and only if $\varphi_1(X)+\varphi_2(Y) + \varphi_3(Z) = 0$, for some smooth functions $\varphi_1$, $\varphi_2$ and $\varphi_3$ with smooth inverses.
Then, since $\varphi_3$ has a smooth inverse on $I_3$, 
we can write $\psi(t)  = \varphi_3^{-1}(-t)$,
so that $F(X,Y,Z) = 0$ is equivalent to  $Z = \psi(\varphi_1(X) + \varphi_2(Y))$.
On the other hand, $F(X,Y,Z) = 0$ can be rearranged into the form
\[
Z = \left(\frac{1}{1-\lambda}\right)^2 (X-Y)(Y-\lambda X),
\] so there is an open set $U\subset I_1\times I_2$ on which we have
\[ \psi(\varphi_1(X) + \varphi_2(Y)) = \left(\frac{1}{1-\lambda}\right)^2 (X-Y)(Y-\lambda X).\]
Write 
\[f(X,Y)=\left(\frac{1}{1-\lambda}\right)^2 (X-Y)(Y-\lambda X).
\]
Lemma \ref{lem:difftest} then informs us that
\begin{equation} \label{calculation}
\frac{\partial^2\left(\log|f_X/f_Y|\right)}{\partial X\partial Y}
\end{equation}
is identically zero on $U$. We can calculate \eqref{calculation} directly and obtain a contradiction. Indeed,
\[\log|f_X/f_Y| = \log\left|\frac{-2\lambda X +Y+\lambda Y}{-2 Y +X+\lambda X}\right|
= \log|-2\lambda X +(1+\lambda) Y| - \log|-2Y +(1+\lambda) X|,\]
and so
\begin{align*}\frac{\partial^2\left(\log|f_X/f_Y|\right)}{\partial X\partial Y}
&= \frac{\partial}{\partial X} \left(\frac{1+\lambda}{-2\lambda X +(1+\lambda) Y} + \frac{2}{-2Y +(1+\lambda) X}\right)
\\&=\frac{2\lambda(1+\lambda)}{(-2\lambda X +(1+\lambda) Y)^2} - \frac{2(1+\lambda)}{(-2Y +(1+\lambda) X)^2}.
\end{align*}
A rearrangement of the latter expression is
\begin{equation} \label{frac}
[2(1+\lambda)(4\lambda-(1+\lambda)^2)] \frac{Y^2-\lambda X^2}{(-2\lambda X +(1+\lambda) Y)^2(-2Y +(1+\lambda) X)^2}.
\end{equation}
The assumption that $\lambda \neq 1,-1$ means that the expression is square brackets is non-zero. Therefore, \eqref{frac} equals zero only when $Y^2= \lambda X^2$, so it does not vanish on any nontrivial open set.  
This gives the required contradiction, which completes the proof of the claim and also the theorem.
\end{proof}

Note that in this proof we did not use the hypothesis that $\lambda \neq 0$. In fact, Theorem \ref{thm:main} holds for the case $\lambda =0$, if we define $A+_G 0B$ to be the set $\{a+0b: (a,b) \in G \}= \linebreak \{a : \exists b, (a,b) \in G \}$.

\section{Small $\max \{ |A+_G A|, |A+_G \lambda A|, |A \cdot_G A|\}$} \label{sec:construction}

The following construction shows that the exponent $6/11$ in Theorem \ref{thm:main} cannot be improved beyond $3/4$. The construction is taken from a recent paper of Alon, Ruzsa and Solymosi \cite{ARS2}, where it was used to illustrate limitations to the growth of $\max \{|A+_GA|, |A\cdot_G A| \}$. It applies verbatim to the problem of bounding $\max \{|A+_GA|, |A+_G \lambda A|, |A\cdot_G A| \}$, and we include the details for the completeness of this paper.

\begin{Theorem}
For any $n \in \mathbb N$ and $\lambda \in \mathbb Z$, there exists a set $A \subset \mathbb R$ with $|A|=n$ and $G \subset A \times A$ 
such that $|G| \gtrsim n^{8/5}$ and
\begin{equation} \label{con1}
\max \{|A+_G A|, |A+_G \lambda A|, |A \cdot_G A|\} \lesssim n^{6/5}.  
\end{equation}
In particular,
\begin{equation} \label{con2}
\max \{|A+_G A|, |A+_G \lambda A|, |A \cdot_G A| \} \lesssim |G|^{3/4}.
\end{equation}

\end{Theorem}

\begin{proof}
Let $n \in \mathbb N$ and define
\[
A:=\left \{ \frac{u \sqrt v}{\sqrt w}   : u,v,w \text{ are distinct primes, } v,w \leq n^{1/5}, u \leq n^{3/5} \right \}.
\]
Note that $\frac{u \sqrt v}{\sqrt w} = \frac{u' \sqrt v'}{\sqrt w'}$ if and only if $(u,v,w)=(u',v',w')$, and so $|A| \gtrsim n$.
Define
\[
G:= \left\{\left(\frac{u\sqrt v}{ \sqrt w},\frac{z\sqrt w}{ \sqrt v}\right ) : v,w,u,z \text{ are distinct primes, } v,w \leq n^{1/5}, u,z \leq n^{3/5} \right \}
\]
and observe that $|G| \gtrsim n^{8/5}$. It remains to check that $|A+_G A|, |A+_G \lambda A|,|A \cdot_G A| \lesssim n^{6/5}$. 
\begin{itemize}
    \item Note that \[
    A+_GA \subset \left\{ \frac{uv+zw}{\sqrt{vw}} : u,v,w,z \in \mathbb N, v,w \leq n^{1/5}, u,z \leq n^{3/5} \right\}.
    \]
    There are at most $2n^{4/5}$ possible values for the numerator and at most $n^{2/5}$ possible values for the denominator. Thus $|A+_G A| \ll n^{6/5}$.
    \item As above, \[
    A+_G \lambda A \subset \left\{ \frac{uv+\lambda zw}{\sqrt{vw}} : u,v,w,z \in \mathbb N, v,w \leq n^{1/5}, u,z \leq n^{3/5} \right\},
    \]
    and thus $|A+_G \lambda A| \ll_{\lambda} n^{6/5}$.
    \item We have 
    \[
    A \cdot_G A \subset \{uz: u,z \in \mathbb N, u,z \leq n^{3/5} \}
    \]
    and so $|A \cdot_G A | \leq n^{6/5}$.
\end{itemize}
\end{proof}

\section{Other sum-product type questions with three sets along graphs} \label{sec:variants}

\subsection{Products cannot be replaced with ratios}

The statement of Theorem \ref{thm:main} is rather sensitive to apparently small changes. The next result shows that we cannot replace the product set with the ratio set and still obtain a non-trivial result. Given $A,B \subset \mathbb R$ and $G \subset A \times B $, define the ratio set of $A$ and $B$ along $G$ to be
\[
A /_G B:= \{ a/b : (a,b) \in G \}.
\]

\begin{Theorem} \label{thm:cons2}
For any $n \in \mathbb N$ and $\lambda \in \mathbb R \setminus \{0\}$, there exists $G \subset X \times Y$ with $|G|=n^2$
such that 
\begin{equation} 
\max \{|X+_G Y|, |X+_G \lambda Y|, |X /_G Y|\} \ll n.  
\end{equation}
In particular, $\max \{|X+_G Y|, |X+_G \lambda Y|, |X /_G Y|\} \ll |G|^{1/2}$.
\end{Theorem}

\begin{proof}
Let $A=\{2^{i} : i \leq n\}$ and $P=A \times A$. In the projective plane, we have three pencils of $O(n)$ lines which cover $P$; the pencil of horizontal lines (intersecting at a common point on the line at infinity) which we label as $\mathcal L_1$, the pencil of vertical lines, labelled $\mathcal L_2$, and the pencil of lines through the origin, labelled $\mathcal L_3$.

For any two triples $p_1,p_2,p_3$ and $q_1,q_2,q_3$ of non-collinear points in the projective plane, there exists a projective transformation $\pi$ such that $\pi(p_i)=q_i$ for $i=1,2,3$. We can thus find a projective transformation $\pi$ such that
\begin{itemize}
    \item $\pi(\mathcal L_1)$ is a pencil of parallel lines with slope $-1$,
    \item $\pi(\mathcal L_2)$ is a pencil of parallel lines with slope $-1/\lambda$,
    \item $\pi(\mathcal L_3)$ is (still) a pencil of  lines through the origin.
\end{itemize}
Moreover, since projective transformations preserve incidence structure, each of these three pencils of size $O(n)$ covers the set $\pi(P)$. Label $G:=\pi(P)$ and note that $G \subset X \times Y$ for some sets $X,Y \subset \mathbb R$.

The elements of the pencil $\pi(\mathcal L_1)$ are in bijection with $X+_G Y$, and so $|X+_G Y| \leq n$. Similarly, there is a bijection from $\pi(\mathcal L_2)$ to $X+_G \lambda Y$ and from $\pi(\mathcal L_3)$ to $X /_G Y$.

\end{proof}

By computing the projective transformation $\pi$ in the previous proof explicitly, we can also explicitly describe the construction of the sets $X$, $Y$ and $G$ given by Theorem \ref{thm:cons2}. Define
\begin{align*}
X&=\{2^{i}+\lambda2^j :i,j \in \mathbb N, i,j \leq n \},
\\ Y & =\{-2^{k}-2^l :k,l \in \mathbb N, k,l \leq n \},
\\ G&=\{ (2^{i}+\lambda2^j, -2^{i}-2^j) : i,j \in \mathbb N, i,j \leq n\}.
\end{align*}
Note that $|G|=n^2$ and $|X+_G Y|, |X+_G \lambda Y| = n=|G|^{1/2}$. Meanwhile,
\[
X/_G Y= \left \{-\frac{1+\lambda2^{j-i} }{1+2^{j-i}} : i,j \in \mathbb N, i,j \leq n\right \}.
\]
There are $2n-1$ values of $j-i$, and so $|X/_G Y| \leq 2n-1 \ll |G|^{1/2}$.

\subsection{A variant of Theorem \ref{thm:main}}

Theorem \ref{thm:cons2} shows the limitations to possible generalisations of Theorem \ref{thm:main}. On the other hand, there do exist some natural sum-product type statements which can be derived by small modifications of the proof of Theorem \ref{thm:main}. We present once such an example here.

\begin{Theorem} \label{thm:products}
Let $A, B \subset \mathbb R$ and $G \subset A \times B$. Let $\alpha, \beta \in \mathbb R$ with $\alpha \neq \beta$. Then
\begin{equation} \label{prodeq}
\max\{ |A+_G B|, |A\cdot_G B|, |(A+\alpha) \cdot_G (B+ \beta)| \} \gg |G|^{6/11}
\end{equation}
\end{Theorem}

We omit the proof of Theorem \ref{thm:products} since it largely repeats the arguments used to prove Theorem \ref{thm:main}. The main difference is that the polynomial $F(X,Y,Z)$ in the proof is more complicated than the corresponding polynomial appearing in the proof of Theorem \ref{thm:main}, and so is the task of checking the non-degeneracy conditions for this $F$.

All three sets are needed to get a non-trivial bound in Theorem \ref{thm:products}. Consider the following example, in the same spirit as the construction from section \ref{sec:picture}, for which
\[
\max\{ |A\cdot_G B|, |(A+\alpha) \cdot_G (B+ \beta)| \} < |G|^{1/2}.
\]
Let $C,D \subset \mathbb R$ be arbitrary sets of cardinality $n$, and consider two sets of hyperbolae
\[
H_1:= \{ xy= c : c \in C \},\,\,\,\,\,\, H_2:= \{(x+\alpha)(y+\beta)=d: d \in D \}.
\]
Let $G$ be the set of all intersection points of two hyperbolas, one from each of the two families. For suitable choices of $C$ and $D$, all of these pairs of hyperbolas intersect in two points, and so $|G|=2n^2$. We have $G \subset A \times B$ for some $A,B \subset \mathbb R$. By construction,
\[
A \cdot_G B = C, \,\,\,\text{and}\,\,\,\, (A+\alpha) \cdot_G (B+\beta) =D,
\]
so that $|A \cdot_G B|, |(A+\alpha) \cdot_G (B+\beta)| = \frac{1}{\sqrt 2}|G|^{1/2}$.

On the other hand, the condition in Theorem \ref{thm:products} that $\alpha \neq \beta$ may not be necessary. It is conceivable that the more relaxed condition that at least one of $\alpha$ or $\beta$ is non zero is sufficient to reach the conclusion \eqref{prodeq} of Theorem \ref{thm:products}. 

\section{Just sums and products} \label{sec:SP}

There exist unconditional non-trivial lower bounds for $\max\{|A+_G A|,|A \cdot_G A| \}$ for sufficiently dense graphs $G$.  Alon, Ruzsa and Solymosi \cite{ARS} proved that
\begin{equation} \label{SPgraphs}
\max \{|A+_G A|, |A\cdot_G A|\} \gg \frac{|G|^{3/2}}{|A|^{7/4}},
\end{equation}
and it follows from this that $\max \{|A+_G A|, |A\cdot_G A|\} \gg |G|^{1/2+\epsilon}$ provided that $|G| \gg |A|^{7/4+\epsilon'}$. The argument in \cite{ARS} is a modification of Elekes's \cite{E} proof of a sum-product estimate using a single application of the Szemer\'{e}di-Trotter Theorem.

We can also use Theorem \ref{11/6ES} to give an alternative bound for $\max\{ |A+_G A|, |A \cdot_G A|\}$. In the case when $A=B$, the following result gives an improvement to \eqref{SPgraphs} when $|G|=o(|A|^{11/6})$, and extends the range of sizes of $G$ for which we can obtain a non-trivial bound $\max \{|A+_G B|, |A \cdot_G B| \} \geq |G|^{1/2+c}$.

\begin{Theorem} \label{thm:SP}
Let $A,B \subset \mathbb R$ and let $G \subset A \times B$. 
Then
\begin{equation} \label{ourSP}
\max \{|A+_G B|, |A \cdot_G B| \}\gg \frac{|G|^{3/4}}{|A|^{3/8}}.
\end{equation}
In particular, for any $\epsilon>0$,
\begin{equation} \label{nontrivial}
|G| \geq |A|^{3/2+ \epsilon} \Rightarrow \max \{|A+_G B|, |A \cdot_G B| \}\gg |G|^{1/2+ \epsilon'},
\end{equation}
where $\epsilon'=\frac{4\epsilon}{24+16\epsilon}$.

\end{Theorem}

\begin{proof}
The statement of $\eqref{nontrivial}$ follows from \eqref{ourSP}, and hence it suffices to prove \eqref{ourSP}. Also, we may assume that $|G| \geq c|A|$, where $c$ is a (sufficiently large) fixed constant. This is because the bound \eqref{ourSP} is worse than trivial when $|G| \ll |A|$, and so certainly true.

Write
\[
C:=A+_G B,\,\,\,\, D:=A \cdot_G B.
\]
We will double count solutions to the system of equations
\begin{align}
c&=a+b \label{eq7}
\\ d&=ab, \label{eq8}
\end{align}
such that $(a,b) \in G$ and $(c,d) \in C \times D$. Define $S$ to be the number of solutions to this system. Once again,
\begin{equation} \label{easy3}
S=|G|,
\end{equation}
since each pair $(a,b) \in G$ gives rise to a unique solution.

We can eliminate $b$ from the system above, and deduce that every contribution $(a,b,c,d)$ to $S$ gives rise to a solution $(a,c,d)$ to the equation $F(a,c,d)=0$, where
\[
F(X,Y,Z)=X(Y-X)-Z.
\]
We thus have $S \leq |Z(F) \cap A \times C \times D|$. Note that $F$ is irreducible, since $Z$ appears only as an isolated linear term. We claim that $F$ is non-degenerate. Assuming the claim, it then follows from Theorem \ref{11/6ES} and \eqref{easy3} that
\begin{equation} \label{twoterms}
|G| \ll |A|^{1/2}|C|^{2/3}|D|^{2/3}+|A|+|C||A|^{1/2}+|D||A|^{1/2}.
\end{equation}
If the first term from the right hand side of \eqref{twoterms} is dominant then 
\[
\max\{|C|,|D|\} \gg \frac{|G|^{3/4}}{|A|^{3/8}},
\]
as required. The assumption that $|G| \geq c|A|$ implies that $|A|$ cannot dominate the right-hand side. We are then left with the case whereby one of the last two terms of \eqref{twoterms} dominates. But then
\[
\max \{ |C|, |D| \} \gg \frac{|G|}{|A|^{1/2}} \gg \frac{|G|^{3/4}}{|A|^{3/8}},
\]
where the last inequality follows (with room to spare) from the assumption that $|G| \geq c|A|$.

It remains to check that $F$ is non-degenerate. Define $f(X,Y)=X(Y-X)$. As in the proof of Theorem \ref{thm:main}, we need to show that
\[
\frac{\partial^2\left(\log|f_X/f_Y|\right)}{\partial X\partial Y}
\]
is not identically zero on any open set. It can be calculated that
\[
\frac{\partial^2\left(\log|f_X/f_Y|\right)}{\partial X\partial Y}=\frac{2}{(Y-2X)^2}.
\]
This is never zero, and therefore $F$ is indeed non-degenerate.
\end{proof}

\section*{Acknowledgements}

The author was supported by the Austrian Science Fund FWF Project P 30405-N32. I am very grateful to Sophie Stevens for carefully reading an earlier draft and giving several suggestions which have improved the paper. I would like to thank Audie Warren for helpful conversations which helped towards proving Theorem \ref{thm:SP}. I am also grateful to Mehdi Makhul for helpful discussions.

\end{document}